\documentclass{amsart}
\usepackage{amssymb}
\usepackage{graphicx}
\usepackage{enumerate}
\usepackage{multirow}
\usepackage{amsmath,color}
\usepackage{hyperref}
\usepackage{url}

\newtheorem{thm}{Theorem}[section]
\newtheorem{cor}[thm]{Corollary}
\newtheorem{lem}[thm]{Lemma}
\newtheorem{prop}[thm]{Proposition}

\numberwithin{equation}{section}
\newcommand{\seq}[1]{\langle #1\rangle}

\title{Distance-regular Cayley graphs \linebreak with small valency}

\author{Edwin R. van Dam}
\address{Department of Econometrics and O.R., Tilburg University, The Netherlands}
\email{Edwin.vanDam@uvt.nl}
\author{Mojtaba Jazaeri}
\address{Department of Mathematics, Shahid Chamran University of Ahvaz, Ahvaz, Iran}
\address{School of Mathematics, Institute for Research in Fundamental Sciences (IPM), P.O. Box: 19395-5746, Tehran, Iran}
\email{M.Jazaeri@scu.ac.ir, M.Jazaeri@ipm.ir}

\DeclareMathOperator{\Cay}{Cay}

\DeclareMathOperator{\tr}{tr}
\DeclareMathOperator{\mud}{mod}
\def\G{\Gamma}

\begin{document}

\subjclass[2010]{05E30}

\keywords{Cayley graph; Distance-regular graph}

\begin{abstract} We consider the problem of which distance-regular graphs with small valency are Cayley graphs.
We determine the distance-regular Cayley graphs with valency at most $4$, the Cayley graphs among the  distance-regular graphs with known putative intersection arrays for valency $5$, and the Cayley graphs among all distance-regular graphs with girth $3$ and valency $6$ or $7$. We obtain that the incidence graphs of Desarguesian affine planes minus a parallel class of lines are Cayley graphs. We show that the incidence graphs of the known generalized hexagons are not Cayley graphs, and neither are some other distance-regular graphs that come from small generalized quadrangles or hexagons. Among some ``exceptional'' distance-regular graphs with small valency, we find that the Armanios-Wells graph and the Klein graph are Cayley graphs.
\end{abstract}

\maketitle

\section{Introduction}
The classification of distance-regular Cayley graphs is an open problem in the area of algebraic graph theory \cite[Problem 71-(ii)]{DKT}. Partial results have been obtained by Abdollahi and the authors \cite{ADJ}, Miklavi\v{c} and Poto\v{c}nik \cite{MP1,MP2}, and Miklavi\v{c} and \v{S}parl \cite{MS}, among others.

Here we focus on distance-regular graphs with small valency.
It is known that there are finitely many distance-regular graphs with fixed valency at least $3$ \cite{BDKM}. In addition, all distance-regular graphs with valency $3$ are known (see \cite[Thm.~7.5.1]{BCN}), as are all intersection arrays for distance-regular graphs with valency $4$ \cite{BK}. There is however no complete classification of distance-regular graphs with fixed valency at least $5$. It is believed though that every distance-regular graph with valency $5$ has intersection array as in Table \ref{tabledrgvalency5}. Besides these results, all intersection arrays for distance-regular graphs with girth $3$ and valency $6$ or $7$ have been determined. We therefore study the problem of which of these distance-regular graphs with small valency are Cayley graphs.

After some preliminaries in Section \ref{sec:pre}, we study several families of distance-regular graphs that have members with small valency. Several of the results in this section are standard. Besides these standard results, we obtain in Proposition \ref{prop:affineplane} that the incidence graphs of the Desarguesian affine planes minus a parallel class of lines are Cayley graphs. In Section \ref{generalizedpolygons}, we study generalized polygons. By extending a known method for generalized quadrangles, we are able to prove (among other results) that the incidence graphs of all known generalized hexagons are not Cayley graphs; see Proposition \ref{IG(GH)}. Moreover, we show that neither are some other distance-regular graphs that come from small generalized quadrangles or hexagons.

We then determine all distance-regular Cayley graphs with valency $3$ and $4$ in Sections \ref{sec:3} and \ref{sec:4}, respectively. Next, we characterize in Section \ref{sec:5} the Cayley graphs among the distance-regular graphs with valency $5$ with one of the known putative intersection arrays. Most of our new results (besides the above mentioned ones) are negative, in the sense that we prove that certain distance-regular graphs are not Cayley graphs. However, we surprisingly do find that the Armanios-Wells graph is a Cayley graph. This gives additional, previously unknown, information about the structure of this distance-transitive graph on $36$ vertices, as we remark after Proposition \ref{Armanios-Wells}.

In the final section, we consider distance-regular graphs with girth $3$ and valency $6$ or $7$. Most of these graphs have been discussed in earlier sections. As another exception, we obtain that the Klein graph on $24$ vertices is a Cayley graph.

\section{Preliminaries}\label{sec:pre}
All graphs in this paper are undirected and simple, i.e., there are no loops or multiple edges. A connected graph $\Gamma$ is called distance-regular with diameter $d$ and intersection array $$\{b_{0},b_{1},\ldots,b_{d-1};c_{1},c_{2},\ldots,c_{d}\}$$ whenever for every pair of vertices $x$ and $y$ at distance $i$, the number of neighbors of $y$ at distance $i-1$ from $x$ is $c_{i}$ and the number of neighbors of $y$ at distance $i+1$ from $x$ is $b_{i}$, for all $i=0,\dots,d$. It follows that a distance-regular graph is regular with valency $k=b_{0}$. The number of neighbors of $y$ at distance $i$ from $x$ is denoted by $a_{i}$, and $a_{i}=k-b_{i}-c_{i}$. The girth of a distance-regular graph follows from the intersection array. The odd-girth (of a non-bipartite graph) equals the smallest $i$ for which $a_i>0$; the even-girth equals the smallest $i$ for which $c_i>1$. A distance-regular graph is called antipodal if its distance-$d$ graph is a disjoint union of complete graphs. This property follows from the intersection array.

A distance-regular graph with diameter $2$ is called strongly regular. A strongly regular graph with parameters $(n,k,\lambda,\mu)$ is a $k$-regular graph with $n$ vertices such that every pair of adjacent vertices has $\lambda$ common neighbors and every pair of non-adjacent vertices has $\mu$ common neighbours. Thus, $\lambda=a_1$, $\mu=c_2$, and the intersection array is $\{k,k-1-\lambda;1,\mu\}$. For more background on distance-regular graphs, we refer to the monograph \cite{BCN} or the recent survey \cite{DKT}.

Let $G$ be a finite group and $S$ be an inverse-closed subset of $G$ not containing the identity element $e$ of $G$. Then the (undirected) Cayley graph $Cay(G,S)$ is a graph with vertex set $G$ such that two vertices $a$ and $b$ are adjacent whenever $ab^{-1} \in S$. Recall that all Cayley graphs are vertex-transitive and a Cayley graph $Cay(G,S)$ is connected if and only if the subgroup generated by $S$, which is denoted by $\seq{S}$, is equal to $G$. Following Alspach \cite{BW}, the subset $S$ in $Cay(G,S)$ is called the connection set. It is well-known that a graph $\G$ is a Cayley graph if and only if it has a group of automorphisms $G$ that acts regularly on the vertices of $\G$.

The commutator of two elements $a$ and $b$ in a group $G$ is denoted by $[a,b]$. Furthermore, the center of $G$ is denoted by $Z(G)$.


\subsection{Halved graphs}

The following observation is straightforward but very useful. Let $\Gamma$ be a Cayley graph $Cay(G,S)$ with diameter $d$. Define sets $S_i$ recursively by $S_{i+1}=SS_{i} \setminus (S_{i} \cup S_{i-1})$ for $i=2,\dots,d$, where $S_1=S$ and $S_{0}=\{e\}$. Then the distance-$i$ graph $\Gamma_i$ of $\Gamma$ is again a Cayley graph, $Cay(G,S_{i})$. In particular, when $\Gamma$ is bipartite, then its halved graphs (the components of $\Gamma_2$) are Cayley graphs.

\begin{lem} \label{distance-i}
The distance-$i$ graph of a Cayley graph $\Gamma$ with diameter $d$ is again a Cayley graph, for $i=2,\dots,d$. Also the halved graphs of $\Gamma$ are Cayley graphs.
\end{lem}

Clearly, also the complement $\overline{G}$ of a Cayley graph $G$ is a Cayley graph.

\subsection{Large girth}

In the later sections we will see many distance-regular graphs with large girth. The following lemmas will then turn out to be useful.

\begin{lem} \label{abelian}
Let $\Gamma$ be a Cayley graph $Cay(G,S)$ with girth $g$, where $|S|>2$. If $G$ is abelian, then $g\leq 4$ and $\G$ contains a --- not necessarily induced --- $4$-cycle.
\end{lem}

\begin{proof}
Let $a$ and $b$ be in $S$ such that $a \neq b^{-1}$. Then $e \sim a \sim ba=ab \sim b \sim e$, so $\G$ contains a $4$-cycle, and hence $g \leq 4$.
\end{proof}

\begin{lem} \label{Element order-Connection set}
Let $\Gamma$ be a Cayley graph $Cay(G,S)$ with girth $g>4$. Suppose that $S$ contains an element of order $m$, with $m>2$. Then $g \leq m$ and the vertices of $\G$ can be partitioned into induced $m$-cycles.
\end{lem}
\begin{proof}
Suppose $a \in S$ has order $m>2$. Then $b \sim ab \sim a^{2}b \sim \cdots \sim a^{m-1}b \sim b$, for every $b\in G$. Now suppose that this $m$-cycle is not induced. Then it follows that there is an $i$, with $1<i<m-1$, such that $a^{i} \in S$. But then $b \sim a^{i}b \sim a^{i+1}b \sim ab \sim b$, which contradicts the assumption that $g>4$.
So every vertex is in an induced $m$-cycle, and the result follows.
\end{proof}

Note that the above partition of vertices into $m$-cycles is the same as the partition of $G$ into the right cosets of the cyclic subgroup $H$ generated by $a$.

In general, if $\G$ is a Cayley graph $Cay(G,S)$, and $H$ is a subgroup of $G$, then the induced subgraph on each of the right cosets of $H$ is regular, and all these subgraphs are isomorphic to each other.

\subsection{Normal subgroups and equitable partitions}\label{sec:normal}

If $\G$ is a Cayley graph $Cay(G,S)$ and $H$ is a normal subgroup of $G$, then the partition into the (distinct) cosets $Hc$ is equitable, in the sense that each vertex in $Hc$ has the same number of neighbors in $Hb$, for each $c$ and $b$. This number is easily shown to be $|S \cap Hcb^{-1}|$. The quotient matrix $Q$ of the equitable partition contains these numbers, i.e. $Q_{Hc,Hb}=|S \cap Hcb^{-1}|$. It is well-known and easy to show (by ``blowing up'' eigenvectors \cite[Lemma 2.3.1]{BH}) that each eigenvalue of $Q$ is also an eigenvalue of $\G$. We will use this fact in some of the later proofs, for example to show that the Biggs-Smith graph is not a Cayley graph.

Note also that the quotient matrix is in fact the adjacency matrix of a Cayley multigraph on the quotient group $G/H$, with connection multiset $S/H=\{Hs \mid s \in S\}$. When $\G$ is an antipodal distance-regular (Cayley) graph with diameter $d$, then it is easy to show that $N_d=S_d \cup \{e\}$ is a subgroup of $G$. If this group is normal, then it follows that there is a Cayley graph over the quotient group $G/N_d$ with connection set $\{N_ds \mid s \in S\}$ (cf. \cite[Lemma 2.2]{MP2}). This quotient graph is the folded graph of $\G$, and it is well-known to be distance-regular, too.


\subsection{Dihedral groups}

Miklavi\v{c} and Poto\v{c}nik \cite{MP1,MP2} classified the distance-regular Cayley graphs over a cyclic or dihedral group. They already observed in \cite{MP1} that a primitive distance-regular graph over a dihedral group must be a complete graph. In \cite{MP2}, they moreover showed the following.

\begin{prop}\cite{MP2}\label{dihedralclassification}
A distance-regular Cayley graph over a dihedral group must be a cycle, complete graph, complete multipartite graph, or the bipartite incidence graph of a symmetric design.
\end{prop}
We will see these graphs also in Section \ref{sec:some}. More importantly, we will use this classification in some of the results in the later sections.

\subsection{Erratum}\label{sec:erratum}

In \cite{ADJ}, we claimed that in the distance-regular line graph $\G$ of the incidence graph of a generalized $d$-gon of order $(q,q)$, any induced cycle is either a triangle or a $2d$-cycle. This is not correct however. Instead, every induced cycle in $\G$ is either a $3$-cycle or an even cycle of length at least $2d$. Consequently, Theorem 3.1. in \cite{ADJ} may not be correct. Instead, we have the following result.

\begin{thm} \label{lineHK}
Let $d \geq 2$, let $\Gamma$ be the line graph of the incidence graph of a generalized $d$-gon of order $(q,q)$, and suppose that $\Gamma$ is a Cayley graph $\Cay(G,S)$. Then there exist two subgroups $H$ and $K$ of $G$ such that $S=(H \cup K) \setminus \{e\}$, with $|H|=|K|=q+1$ and $H \cap K=\{e\}$ if and only if $\seq{a} \subseteq S \cup \{e\}$ for every element $a$ of order $2i$ in $S$, with $i \geq d$.
\end{thm}

The correction of the above result has no impact on the validity of the following result in \cite[Prop.~3.4]{ADJ}. In fact, by Lemma \ref{abelian}, the proof can do without the above theorem.

\begin{prop} \label{tuttecoxeter}
The line graph of Tutte's $8$-cage is not a Cayley graph.
\end{prop}

\begin{proof} Let $\G$ be the line graph of Tutte's $8$-cage, and suppose that it is a Cayley graph $Cay(G,S)$.
Then $|G|=45$ and $|S|=4$. By Lemma \ref{abelian}, $G$ cannot be abelian because $\G$ has no $4$-cycles. But all groups of order $45$ are abelian, so we have a contradiction.
\end{proof}

\section{Some families of distance-regular graphs}\label{sec:some}

It is clear that the cycle $C_n$ is a distance-regular Cayley graph over the cyclic group. Thus, every distance-regular graph with valency $2$ is a Cayley graph.
Here we mention some other relevant families of distance-regular graphs with members of small valency.

\subsection{Complete graphs, complete multipartite graphs, and complete bipartite graphs minus a matching}\label{complete bipartite minus matching}
The complete graph $K_n$ and the regular complete multipartite graph $K_{m \times n}$ are distance-regular Cayley graphs (with diameters 1 and 2, respectively). Indeed, $K_n$ is a Cayley graph over any group of order $n$, whereas $K_{m \times n}$ (with $m$ parts of size $n$) is a Cayley graph over the cyclic group $\mathbb{Z}_{mn}$, with connection set $S=\mathbb{Z}_{mn}\setminus m\mathbb{Z}_{mn}$. Note that the complete bipartite graph $K_{2 \times n}$ is usually denoted by $K_{n,n}$.

A complete bipartite graph $K_{n,n}$ minus a complete matching, which is denoted by $K^{*}_{n,n}$, is distance-regular with valency $n-1$ and diameter $3$. Even though it may be clear that this is also a Cayley graph, we will describe it as such explicitly. Indeed, let $D_{2n}=\seq{a,b \mid a^{n}=b^{2}=1,bab=a^{-1}}$. Then the Cayley graph $Cay(D_{2n},S)$, where $S=\{ba^{i} \mid 1\leq i \leq n-1\}$ is the complete bipartite graph $K_{n,n}$ minus a complete matching, with two bipartite parts $\seq{a}$ and $b\seq{a}$. This graph can also be described as the incidence graph of a symmetric design; see Section \ref{symmetricdesign}.

\subsection{Paley graphs}\label{Paley}

The Paley graphs are defined as Cayley graphs. Let $q$ be a prime power such that $q \equiv 1~(\mud 4)$. Let $G$ be the additive group of $GF(q)$ and let $S$ be the set of nonzero squares in $GF(q)$. Then the Paley graph $P(q)$ is defined as the Cayley graph $Cay(G,S)$. It is distance-regular with diameter $2$ and valency $(q-1)/2$.

\subsection{Hamming graphs, cubes, and folded cubes}\label{Cubes}
The Hamming graph $H(d,q)$ is the $d$-fold Cartesian product of $K_q$. It can therefore be described as a Cayley graph over (for example) $\mathbb{Z}_q^d$ with the set of vectors of (Hamming) weight one as connection set. It is distance-regular with valency $d(q-1)$ and diameter $d$.

The Hamming graph $H(2,q)$ is also known as the lattice graph $L_2(q)$. The Shrikhande graph is a distance-regular graph with the same intersection array as $L_2(4)$, and it is a Cayley graph $\Cay(\mathbb{Z}_{4} \times \mathbb{Z}_{4}, \{\pm (0,1), \pm (1,0), \pm (1,1)\})$. A Doob graph is a Cartesian product of Shrikhande graphs and $K_4$'s. These Doob graphs are thereby distance-regular Cayley graphs as well.

The Hamming graph $H(d,2)$ is also known as the $d$-dimensional (hyper)cube graph $Q_d$.
The folded $d$-cube can be obtained from $Q_{d-1}$ by adding a perfect matching connecting its so-called antipodal vertices. This implies that it is a Cayley graph over $\mathbb{Z}_2^{d-1}$ with connection set the set of unit vectors and the all-ones vector. The folded $d$-cube is distance-regular with valency $d$ and diameter $\lfloor d/2 \rfloor$.

\subsection{Odd and doubled Odd graphs}\label{Oddgraph}
The Odd graph $O_{n}$ is the Kneser graph $K(2n-1,n-1)$. It is distance-regular with valency $n$ and diameter $n-1$. Godsil \cite{G} determined which Kneser graphs are Cayley graphs, and it follows that the Odd graph is not a Cayley graph.

The doubled Odd graph $DO_n$ is the bipartite double of the Odd graph $O_{n}$. It is distance-regular with valency $n$ and diameter $d=2n-1$.
It is easy to see that if a graph $\Gamma$ is a Cayley graph $Cay(G,S)$, then its bipartite double is again a Cayley graph over the group $G \times \mathbb{Z}_{2}$ with connection set $S=\{(s,1) \mid s \in S\}$. But the Odd graph is not a Cayley graph, so we cannot apply this argument. Indeed, it turns out that the doubled Odd graph is also not a Cayley graph.
\begin{prop} \label{double Odd}
The doubled Odd graph is not a Cayley graph.
\end{prop}
\begin{proof}
The distance-($d-1$) graph of a doubled Odd graph $DO_n$ (with diameter $d=2n-1$) is a disjoint union of two Odd graphs $O_{n}$. If this graph is a Cayley graph, then its distance-($d-1$) graph is again a Cayley graph, by Lemma \ref{distance-i}. But an Odd graph is not a Cayley graph \cite{G}, so neither is the doubled Odd graph.
\end{proof}

Godsil's results \cite{G} also imply the classification by Sabidussi \cite{SaVTG} of Cayley graphs among the triangular graphs $T(n)$; these are Cayley graphs if and only if $n=2,3,4$ or $n \equiv 3~(\mud 4)$ and $n$ is a prime power.

\subsection{Incidence graphs of symmetric designs}\label{symmetricdesign}

Miklavi\v{c} and Poto\v{c}nik \cite{MP2} \linebreak showed that there is a correspondence between difference sets and connection sets for the incidence graphs of a symmetric design. Recall that a $k$-subset $D$ of a group $G$ of order $n$ is called a $(n,k,\lambda)$ difference set if every nonidentity element $g \in G$ occurs $\lambda$ times among all possible differences $d_1d_2^{-1}$ (we prefer to use multiplicative notation) of distinct elements $d_1$ and $d_2$ of $D$. The development $\{Dg \mid g \in G\}$ of such a difference set is a symmetric 2-$(n,k,\lambda)$ design. 

If $D$ is a difference set in an abelian group $G$, then we can easily construct the incidence graph of its development as a Cayley graph for the group $G \rtimes \mathbb{Z}_{2}$. The elements of this group can be (identified and) partitioned as $G \cup Gc$, where $c^2=1$ and $cgc=g^{-1}$ for all $g \in G$. As a connection set, we take $S=Dc$. It follows that $S$ is inverse closed, and that the corresponding Cayley graph is indeed the incidence graph of the development (a block $Dg$ corresponds to the group element $g^{-1}c$).

Because the Desarguesian projective plane (over $GF(q)$) is a symmetric $2$-$(q^2+q+1,q+1,1)$ design, and can be obtained from a (Singer) difference set in the cyclic group, it follows that the incidence graph of a Desarguesian projective plane is a Cayley graph. It was shown by Loz et al. \cite{LMM} that this Cayley graph is $4$-arc-transitive. We note that all projective planes of order at most $8$ are Desarguesian, and hence all incidence graphs of projective planes with valency at most $9$ are Cayley graphs.

We also note that if $D$ is a difference set in $G$, then the complement $G \setminus D$ is also a difference set in $G$, and its development is the complementary design of the development of $D$. This implies that also the incidence graph of the $2$-$(7,4,2)$ design is a Cayley graph.
Also the $2$-$(11,5,2)$ biplane comes from a difference set (the set of nonzero squares in $\mathbb{Z}_{11}$), so its incidence graph is a Cayley graph. Note that also the (trivial) $2$-$(n,n-1,n-2)$ design comes from a difference set ($D=G\setminus\{e\}$), which gives an alternative proof that $K_{n,n}^*$ is a Cayley graph (see Section \ref{complete bipartite minus matching}).

We denote the incidence graph of a 2-$(n,k,\lambda)$ design by $IG(n,k,\lambda)$. Such a graph is distance-regular with valency $k$ and diameter $3$.

\subsection{Incidence graphs of affine planes minus a parallel class of lines}\label{affineplane}

Similar to the case of symmetric designs, there is a correspondence between certain relative difference sets and connection sets for the incidence graph of an affine plane minus a parallel class of lines. A $k$-subset $R$ of a group $G$ of order $mn$ is called a relative $(m,n,k,\lambda)$ difference set relative to a subgroup $N$ of order $n$ of $G$ if every element of $G \setminus N$ occurs $\lambda$ times among all possible differences $r_1r_2^{-1}$ of elements $r_1$ and $r_2$ of $R$. The development of such a relative difference set is a so-called $(m,n,k,\lambda)$ divisible design. We will not go into the details of the definition of such a divisible design, but restrict to the remark that a $(n,n,n,1)$ divisible design is the same as an affine plane of order $n$ minus a parallel class of lines (for details, see \cite{Pott}). Similar as in Section \ref{symmetricdesign}, if such a divisible design comes from a relative difference set in an abelian group, then its incidence graph is a Cayley graph.

It is known that all Desarguesian planes correspond to relative difference sets, so the incidence graphs of the Desarguesian affine planes minus a parallel class are all Cayley graphs. These include all such distance-regular graphs with valency at most $8$. In particular, for odd prime powers $q$, the set $\{(x,x^2) \mid x \in GF(q)\}$ is a relative difference set in $GF(q)^2$. To include even prime powers, we need a more involved construction of a relative difference set that actually works also for semifields (see \cite[Thm.~4.1]{Pott}). Indeed, if $\mathbb{S}$ is a semifield of order $q$, then we define a group on $\mathbb{S}^2$ using the addition $(x_1,x_2)+(y_1,y_2)=(x_1+y_1,x_2+y_2+x_1y_1)$. In this group, the set $\{(x,x^2) \mid x \in \mathbb{S}\}$ is a relative $(q,q,q,1)$ difference set. We note that if $\mathbb{S}$ is the field on $2^n$ vertices, then the constructed group is isomorphic to $\mathbb{Z}_4^n$.

We denote the incidence graph of a the Desarguesian affine plane of order $q$ minus a parallel class of lines ($pc$) by $IG(AG(2,q)\setminus pc)$. Such a graph is distance-regular with valency $q$ and diameter $4$. We conclude the following.

\begin{prop} \label{prop:affineplane}
For every prime power $q$, the incidence graph of the Desarguesian affine plane of order $q$ minus a parallel class of lines, $IG(AG(2,q)\setminus pc)$, is a Cayley graph.
\end{prop}

\subsection{Generalized polygons}\label{generalizedpolygons}
The incidence graph of a generalized quadrangle or generalized hexagon of order $(q,q)$ is distance-regular with valency $q+1$ and girth $8$ and $12$, respectively. These graphs thus arise in the tables in the following sections. In this section, we will first show, among other results, that for $q \leq 4$, none of these is a Cayley graph. Next to that, we will consider some of the distance-regular line graphs and halved graphs (point graphs) of these graphs.

Indeed, first suppose that the incidence graph $\G$ of generalized polygon of order $(s,s)$ is a Cayley graph. Then its automorphism group contains a subgroup that acts regularly on the vertices of $\G$. It follows that there is an index $2$ subgroup $G$ that acts regularly on both the point set and on the line set, as an automorphism group of the generalized polygon. This situation has been studied by Swartz \cite{Sw} for generalized quadrangles. Using results by Yoshiara \cite{Y} (who exploited an idea of Benson \cite{Benson}; cf. \cite[1.9.1]{PT}), Swartz \cite{Sw} showed that $s+1$ must be coprime to $2$ and $3$. Consequently, we have the following result.

\begin{prop} \label{IG(GQ)}
If the incidence graph of a generalized quadrangle of order $(s,s)$ is a Cayley graph, then $s+1$ is not divisible by $2$ or $3$.
\end{prop}

In particular, it shows that the incidence graphs of generalized quadrangles of orders $(2,2)$ and $(3,3)$ are not Cayley graphs.

We will next derive a similar result for generalized hexagons. The line of proof is the same as for generalized quadrangles. By extracting the main ideas and fine-tuning them, we are able to give a self-contained proof, which in the end even leads to a somewhat stronger result. We note that similar more general techniques and results on generalized hexagons (but not our main results) have also been obtained by Temmermans, Thas, and Van Maldeghem \cite{TTvM}.

As in the above, we assume that the generalized hexagon of order $(s,s)$ has an automorphism group $G$ that acts regularly on points as well as on lines. Thus, the order of $G$ is $(s+1)(s^4+s^2+1)$. We start with a lemma.

\begin{lem}\label{lem:GHp} Let $p=2,3,$ or $5$, and let $g \in G$ be of order $p$. Then $x^g \neq x$ and $x^g$ is not collinear to $x$, for every point $x$.
\end{lem}

\begin{proof} Let $x$ be an arbitrary point. Because $G$ is regular, $g$ fixes no points, and also no lines (otherwise $g=e$) so $x^g \neq x$. In order to show that $x^g$ is not collinear to $x$, we assume that $\ell$ is a line through $x$ and $x^g$, and show that this leads to a contradiction.

If $g$ has order $2$, then $\ell^g$ is a line through $x^g$ and $x^{g^2}=x$, so $\ell^g=\ell$, which is indeed a contradiction.

If $g$ has order $3$, then $x$, $x^g$, and $x^{g^2}$ are pairwise collinear. Similar as in the previous case (order $2$), these three points cannot all be on the line $\ell$, and it follows that they ``generate'' three lines $\ell$, $\ell^g$, and $\ell^{g^2}$. This however gives a $6$-cycle in the incidence graph, which is a contradiction, because its girth is $12$.

Similarly, if $g$ has order $5$, then this gives rise to a $10$-cycle in the incidence graph, which is again a contradiction.
\end{proof}

Note that the case $p=5$ seems specific for generalized hexagons, whereas the cases $2$ and $3$ clearly also apply to generalized quadrangles, because their incidence graphs have girth ``only'' $8$.

Next, we consider the adjacency matrix $A$ of the point graph of the generalized hexagon, and let $M=A+I$. Note that this matrix could also be used to obtain the results for generalized quadrangles. Our matrix $M$ has eigenvalue $s^2+s+1$ with multiplicity one (from the constant eigenvector), $2s$, $0$, and $-s$.
From an automorphism $g$ we make a permutation matrix $Q$, where $Q_{x,y}=1$ if $y=x^g$. Because $g$ is an automorphism, we have that $QA=AQ$, and hence that $QM=MQ$. Using the eigenvalues of $M$, we obtain the following lemma.

\begin{lem}\label{lem:trace} $\tr QM \equiv 1 (\mud s)$.
\end{lem}

\begin{proof} If $g$ has order $n$, then $(QM)^n=Q^nM^n=M^n$. It follows that $QM$ has the same eigenvalues as $M$, possibly multiplied by a root of unity. It has the same eigenvalue $s^2+s+1$ with multiplicity one (from the constant eigenvector) as $M$. For each other eigenvalue, also its conjugates are eigenvalues, and the sum of these is a multiple of the ``original'' eigenvalue $\theta$ of $M$ (because the sum of the relevant roots of unity is integer; for details, see the similar proof for generalized quadrangles by Benson \cite{Benson}). It follows that the sum of all eigenvalues equals $s^2+s+1$ plus integer multiples of $2s$, $0$, and $-s$. Hence $\tr QM \equiv 1 (\mud s)$.
\end{proof}

We can now prove the following.

\begin{prop} \label{IG(GH)}
If the incidence graph of a generalized hexagon of order $(s,s)$ is a Cayley graph, then $s$ is a multiple of $6$ and $s+1$ is not divisible by $5$.
\end{prop}

\begin{proof} Suppose that the incidence graph is a Cayley graph, and that $(s+1)(s^4+s^2+1)$ is divisible by $2$, $3$, or $5$. Then the generalized hexagon has a regular group $G$ of automorphisms, acting regularly on both the point set and the line set. Because the order of this group is divisible by $2$, $3$, or $5$, there is an automorphism $g \in G$ of order $2$, $3$, or $5$. By Lemma \ref{lem:GHp}, $x^g \neq x$ and $x^g$ is not collinear to $x$, for every point $x$. It follows that both $Q$ and $QA$ have zero diagonal, hence $\tr QM=0$. But this contradicts Lemma \ref{lem:trace}, hence $(s+1)(s^4+s^2+1)$ is not divisible by $2$, $3$, or $5$, and this implies that $s$ is a multiple of $6$ and $s+1$ is not divisible by $5$.
\end{proof}

Because generalized hexagons of order $(s,s)$ are only known for prime powers $s$, it follows that all
the incidence graphs of the known generalized hexagons are not Cayley graphs. Note that automorphisms of a putative generalized hexagon of order $(6,6)$ have been studied by Belousov \cite{Belousov}.

Similarly, generalized quadrangles of order $(s,s)$ are only known for prime powers
$s$. Among these known ones, Proposition \ref{IG(GQ)} thus rules out all $s$ except $s=4^i$ (for
$i \in \mathbb{N}$). Among the distance-regular incidence graphs of generalized polygons with
valency at most $5$, we still need to consider the incidence graph of the generalized
quadrangle of order $(4,4)$. For this, we also consider one of the halved graphs, i.e.,
the collinearity (or point) graph.

\begin{prop} \label{IG(GQ(4,4))}
The incidence graph of the generalized quadrangle $GQ(4,4)$ is not a Cayley graph.
\end{prop}
\begin{proof}
Suppose that this bipartite graph $\Gamma$ is a Cayley graph. By Lemma \ref{distance-i}, its
halved graphs are also Cayley graphs. These halved graphs (one of them being the collinearity
graph of the generalized quadrangle) are again distance-regular, with intersection
array $\{20,16;1,5\}$ \cite[Proposition~4.2.2]{BCN}. In other words, it is a strongly
regular graph with parameters $(85,20,3,5)$. By Sylow's
theorem, the only group of order 85 is the cyclic group $\mathbb{Z}_{85}$. Using the properties
of a generalized quadrangle and that the cyclic group is abelian, it is easy to show
that each line (a 5-clique) through $e$ forms a subgroup of $\mathbb{Z}_{85}$, but there is only one
such subgroup, which gives a contradiction, because there are $5$ lines through each point.
\end{proof}

We note that this result also follows from more extensive results by Bamberg and Giudici \cite[Thm.~1.1]{BG} and by Swartz \cite[Thm.~1.3]{Sw}.
We remark that also the result that Tutte's $8$-cage --- the incidence graph of the unique generalized quadrangle of order $(2,2)$ --- is not a Cayley graph, can be obtained using the point graph. The latter is the complement of the triangular graph $T(6)$. Sabidussi \cite{SaVTG} determined the Cayley graphs among the triangular graphs (see also Section \ref{Oddgraph}), and $T(6)$ is not one of them. Thus, Tutte's $8$-cage, also known as the Tutte-Coxeter graph, is not a Cayley graph.

Also Tutte's 12-cage --- the unique incidence graph of a generalized hexagon of order $(2,2)$ --- is not a Cayley graph for an elementary reason, i.e., because it is not vertex-transitive. Note that there are two generalized hexagons of order $(2,2)$, and these are dual, but not isomorphic, to each other. Thus, there are two orbits of vertices in the incidence graph.

We note that similarly there are precisely two generalized quadrangles of order $(3,3)$, and these are dual to each other. This implies that the corresponding incidence graph is not vertex-transitive, and hence this gives another argument for why this graph is not a Cayley graph.

Another argument for why Tutte's 12-cage is not a Cayley graph is obtained by considering the point graphs of the two generalized hexagons of order $(2,2)$. These distance-regular graphs have intersection array $\{6,4,4;1,1,3\}$ and automorphism group $PSU(3,3) \rtimes \mathbb{Z}_{2}$ \cite{drgorg}. If such a graph would be a Cayley graph $Cay(G,S)$, then $G$ must be a subgroup of order $63$ of the above group. Moreover, because the graph has no $4$-cycles, the group must be nonabelian by Lemma \ref{abelian}. However, we checked with {\sf GAP} \cite{GAP} that there are no such subgroups, so we conclude that these graphs are not Cayley graphs. A similar argument applies to the line graph of Tutte's 12-cage, the unique distance-regular graph with intersection array $\{4,2,2,2,2,2;1,1,1,1,1,2\}$. Also this graph has automorphism group $PSU(3,3) \rtimes \mathbb{Z}_{2}$ \cite{drgorg} and no $4$-cycles. Thus, after having checked that there are no nonabelian subgroups of order $189$, we conclude the following.

\begin{prop} \label{CL(GD(2,1))}
The line graph of Tutte's $12$-cage and the point graphs of the two generalized hexagons of order $(2,2)$ are not Cayley graphs.
\end{prop}

Similarly, we can show that the unique distance-regular graph with intersection array $\{6,3,3;1,1,2\}$, the line graph of the incidence graph of the projective plane (generalized $3$-gon) of order $3$ is not a Cayley graph. Indeed, the automorphism group of the incidence graph (and hence of its line graph) is $PSL(3,3) \rtimes \mathbb{Z}_{2}$, and we checked again with {\sf GAP} \cite{GAP} that it has no subgroups of order $52$. We recall from Section \ref{symmetricdesign} that the incidence graph itself is a Cayley graph. We had already observed in \cite[Thm.~5.8]{ADJ} that if the line graph of the incidence graph of a projective plane of small odd order is a Cayley graph, then it should come from a group of both collineations and correlations of the projective plane.

\begin{prop} \label{LinePP3}
The line graph of the incidence graph of the projective plane of order $3$ is not a Cayley graph.
\end{prop}

We next consider the line graph of the incidence graph of the generalized quadrangle of order $(3,3)$.

\begin{prop} \label{LineGQ3}
The line graph of the incidence graph of the generalized quadrangle of order $(3,3)$ is not a Cayley graph.
\end{prop}

\begin{proof} Suppose that this graph $\G$ is a Cayley graph $Cay(G,S)$. Then $G$ is a subgroup of the automorphism group of the incidence graph of the generalized quadrangle that acts regularly on its $160$ flags. It follows that $G$ acts transitively on the point set $P$ and on the line set $L$. Hence $|G_{x}|=|G_{\ell}|=4$ for every $x \in P$ and $\ell \in L$. This implies that for every point (and similarly, for every line), there is an involution in $G$ that fixes it. On the other hand, it is not hard to show that every involution in $G$ fixes either a point or a line, using Benson's results \cite{Benson} or the approach as in Lemma \ref{lem:trace} (see also \cite[Lemma 3.4]{flagquad}).

Now let $H$ be a Sylow $2$-subgroup of $G$. We claim that the intersection of $Z(G)$ and $H$ is trivial. To show this, assume that it is not. Then $H \cap Z(G)$ contains an involution $\sigma$, say, and suppose without loss of generality that $\sigma$ fixes a point $x$, say. Let $\ell$ be a line through $x$ and let $\theta$ be an involution that fixes $\ell$. If $y=x^{\theta}$, then it is easy to see that $\sigma$ also fixes $y$, and hence $\ell$. But then it fixes a flag $(x,\ell)$, which is a contradiction.

Because $Z(G)$ is normal in $G$, it follows that $HZ(G)$ is a subgroup of $G$, with $|HZ(G)|=|H||Z(G)|$. This implies that $|Z(G)|=1$ or $5$. We checked with {\sf GAP} \cite{GAP} that there is no group of order $160$ with $|Z(G)|=5$ and there exists only one group $G$ of order $160$ such that $|Z(G)|=1$; this group is $(\mathbb{Z}_2^4 \rtimes \mathbb{Z}_5) \rtimes \mathbb{Z}_2$.

Now $G$ has a normal subgroup $N=\mathbb{Z}_2^4 \rtimes \mathbb{Z}_5$ of index $2$, and this group does not have any dihedral subgroup, except the ones of order $2$ and $4$. Moreover, the two cosets of $N$ induce an equitable partition of the graph, with quotient matrix of the form $$\begin{bmatrix}m&6-m\\6-m&m\end{bmatrix},$$ with $m=|S \cap N|$. This implies that $\G$ must have an eigenvalue $2m-6$ (besides eigenvalue $6$) and because the integer eigenvalues of $\G$ are $6$, $2$, and $-2$, it follows that $m=2$ or $m=4$.

By Theorem \ref{lineHK} and the fact that $G$ only has elements of orders $1,2,4$, and $5$, it follows that $S=(K_1\cup K_2)\setminus \{e\}$, where $K_1$ and $K_2$ are subgroups of $G$ of order $4$ such that $K_1 \cap K_2 =\{e\}$.

In both the cases $m=2$ and $m=4$, it follows that $S \cap N$ contains involutions $s_1\in K_1$ and $s_2\in K_2$. These two involutions generate a dihedral subgroup of $N$, which implies that this must be the dihedral group of order $4$. But then $s_1$ and $s_2$ commute, and it is clear that $e$ and $s_1s_2$ have at least two common neighbors, while being at distance $2$, and we have a contradiction.
\end{proof}

The last case we will handle in this section is that of the line graph of the incidence graph of a generalized hexagon of order $(3,3)$. Note that it is currently unknown how many such generalized hexagons there are.

\begin{prop} \label{LineGH3}
The line graph of the incidence graph of a generalized hexagon of order $(3,3)$ is not a Cayley graph.
\end{prop}

\begin{proof}
Suppose that this graph $\G$ is a Cayley graph $Cay(G,S)$. Then by the same approach as in the proof of Proposition \ref{LineGQ3}, it follows that $G=(\mathbb{Z}_2^3 \rtimes \mathbb{Z}_7) \times D_{26}$. Again, $G$ has a normal subgroup $N=(\mathbb{Z}_2^3 \rtimes \mathbb{Z}_7) \times \mathbb{Z}_{13}$ of index $2$, and from the eigenvalues of $\G$, we obtain that $m=2$ or $m=4$, where $m=|S \cap N|$.

Observe that $N$ contains seven involutions, which generate an abelian subgroup $\mathbb{Z}_2^3$.
Because $S \cap N$ contains an even number of elements, it also contains an even number of involutions. But these involutions commute and there are no induced $4$-cycles in $\G$, so it easily follows that $S \cap N$ contains no involutions. Because $N$ only has elements of order $1,2,7,13,26,$ and $91$, and $\G$ contains no induces odd-cycles besides triangles, it follows that $S \cap N$ only contains elements of order $26$. Thus, the connection set $S$ has at least two elements of order $26$.

Next, we consider the normal subgroup $K=\mathbb{Z}_2^3 \times D_{26}$, with quotient group $G/K$ isomorphic to $\mathbb{Z}_7$. Note that all elements of order $26$ in $G$ are in $K$, so it follows that $S \cap K$ contains at least two elements. Because the quotient matrix corresponding to the equitable partition of the cosets of $K$ is symmetric and cyclic, it follows that there are essentially only three options; the first row of the quotient matrix must be $[4~ 1~ 0~ 0~ 0~ 0~ 1]$, $[2~ 2~ 0~ 0~ 0~ 0~ 2]$, or $[2~ 0~ 1~ 1~ 1~ 1~ 0]$. All three matrices have eigenvalues of degree $3$ (related to eigenvalues of the $7$-cycle; the roots of $x^3+x^2-2x-1$). But $\G$ has no such eigenvalues, so we have a contradiction.
\end{proof}

Finally, we note that Bamberg and Giudici \cite{BG} claim that none of the classical generalized hexagons and octagons have a group of automorphisms that acts regularly on the points. This implies that none of the point graphs of the known generalized hexagons and octagons are Cayley graphs.

\section{Distance-regular graphs with valency $3$}\label{sec:3}
All distance-regular graphs with valency $3$ are known; see \cite[Thm.~7.5.1]{BCN}. In Table \ref{tabledrgvalency3}, we give an overview of all possible intersection arrays and corresponding graphs, and indicate which of these is a Cayley graph. The latter will follow from the results in the previous section, and the investigations in the current section, as commented in the table. Note that for each intersection array in Table \ref{tabledrgvalency3} there is a unique distance-regular graph.
By $n$, $d$, and $g$, we denote the number of vertices, diameter, and girth, respectively.
\begin{center}
\begin{table}[h]
\begin{tabular}{ l r r c l  c l}
  \hline
  Intersection array & $n$ & $d$ & $g$ & Name & Cayley & Comments\\
  \hline
  \{3;1\}& 4 & 1 & 3&$K_{4}$ & Yes & Sec.~\ref{complete bipartite minus matching}\\
  \{3,2;1,3\} & 6 & 2 &4& $K_{3,3}$ & Yes & Sec.~\ref{complete bipartite minus matching}\\
  \{3,2,1;1,2,3\} & 8 & 3 &4& Cube $ \sim K_{3,3}^*$ & Yes & Sec.~\ref{complete bipartite minus matching}\\
  \{3,2;1,1\} & 10 & 2 &5& Petersen $\sim O_3$ & No & Sec.~\ref{Oddgraph}\\
  \{3,2,2;1,1,3\} & 14 & 3 &6& Heawood $\sim IG(7,3,1)$ & Yes & Sec.~\ref{symmetricdesign}\\
  \{3,2,2,1;1,1,2,3\} & 18 & 4 &6& Pappus $\sim$ & Yes & Prop.~\ref{prop:affineplane}\\
      & & &&\hspace{.5cm}$ IG(AG(2,3)\setminus pc)$ &&\\
  \{3,2,2,1,1;1,1,2,2,3\} & 20 & 5 &6& Desargues $\sim DO_3$ & No & Prop.~\ref{double Odd}\\
  \{3,2,1,1,1;1,1,1,2,3\} & 20 & 5 &5& Dodecahedron & No & Folklore\\
  \{3,2,2,1;1,1,1,2\}& 28 & 4 &7& Coxeter & No & Prop.~\ref{Coxeter} \\
  \{3,2,2,2;1,1,1,3\} & 30 & 4 &8& Tutte's 8-cage $\sim$ & No & Prop.~\ref{IG(GQ)}\\
     & && &\hspace{.5cm}$IG(GQ(2,2))$ &&\\
  \{3,2,2,2,2,1,1,1;& 90 & 8 & 10&Foster & No & Prop.~\ref{Foster}\\
   \hspace{.5cm}1,1,1,1,2,2,2,3\}  & && && &\\
  \{3,2,2,2,1,1,1; & 102 & 7 &9& Biggs-Smith & No & Prop.~\ref{Biggs-Smith}\\
  \hspace{.5cm}1,1,1,1,1,1,3\}  && & & &&\\
  \{3,2,2,2,2,2; & 126 & 6 &12& Tutte's 12-cage $\sim$ & No & Prop.~\ref{IG(GH)}\\
   \hspace{.5cm}1,1,1,1,1,3\}  & && &\hspace{.5cm}$IG(GH(2,2))$ &&\\
  \hline
\end{tabular}
\caption{Distance-regular graphs with valency $3$}\label{tabledrgvalency3}
\end{table}
\end{center}
The first graph in the table that does not occur in the previous section is the dodecahedron. It is however well known that this graph is not a Cayley graph; see for example \cite{fullerene}, where it is shown that the only fullerene Cayley graph is the football (or buckyball) graph.

Also the fact that the Coxeter graph is not a Cayley graph is folklore. In the
literature, e.g., \cite{EK}, it is mentioned as one of the four non-Hamiltonian vertex-
transitive graphs on more than two vertices, and it is noted that none of these four is a Cayley graph. Indeed, the automorphism group of the Coxeter graph is $PGL(2,7)$, and this group has no subgroups of order 28.

\begin{prop} \label{Coxeter}
The Coxeter graph is not a Cayley graph.
\end{prop}

The Foster graph is a bipartite distance-regular graph that can be described as
the incidence graph of a partial linear space that can be considered as a $3$-cover of
the generalized quadrangle of order $(2,2)$. Its halved graphs are distance-regular
with intersection array $\{6,4,2,1;1,1,4,6\}$  (e.g., see \cite[Proposition 4.2.2]{BCN}). The halved graph on the points is the
collinearity graph of this partial linear space.

\begin{prop} \label{Foster}
The Foster graph is not a Cayley graph.
\end{prop}
\begin{proof}
Suppose that the Foster graph is a Cayley graph. By Lemma \ref{distance-i}, its halved
graphs are also Cayley graphs, and these are distance-regular with intersection
array $\{6,4,2,1;1,1,4,6\}$ on $45$ vertices. So suppose that this halved graph is a Cayley graph $Cay(G,S)$, with $G$ of order $45$ and $S$ of size $6$. By Sylow's
theorem, $G$ must be abelian. By Lemma \ref{abelian}, it follows that $\G$ contains a $4$-cycle, which contradicts the fact that both the intersection numbers $a_1$ and $c_2$ are equal to $1$. Thus, a distance-regular
graph with intersection array $\{6,4,2,1;1,1,4,6\}$ cannot be a Cayley graph, and
hence neither can the Foster graph.
\end{proof}

As a side result, we have thus obtained the following.

\begin{cor}\label{halffoster} The collinearity graph of the $3$-cover of the generalized quadrangle $GQ(2,2)$, the unique distance-regular graph with intersection array \linebreak $\{6,4,2,1;1,1,4,6\}$, is not a Cayley graph.
\end{cor}

What remains is to consider the Biggs-Smith graph. The eigenvalues of this graph are very exceptional for a distance-regular graph. It has five distinct irrational eigenvalues, and distinct rational eigenvalues $3,2$, and $0$.

\begin{prop} \label{Biggs-Smith}
The Biggs-Smith graph is not a Cayley graph.
\end{prop}
\begin{proof}
Suppose that the Biggs-Smith graph $\G$ is a Cayley graph $Cay(G,S)$. Then $|G|=102$, so $G$ has a subgroup $H$ of order $51$. It follows that the two cosets of $H$ induce an equitable partition for $\G$. Because $\G$ is connected and not bipartite, the quotient matrix is of the form $$\begin{bmatrix}m&3-m\\3-m&m\end{bmatrix},$$ where $m=1$ or $m=2$. This implies that $\G$ has an eigenvalue $-1$ or $1$, which is a contradiction.
\end{proof}

Now we can conclude this section by the following result.
\begin{thm}
  Let $\Gamma$ be a distance-regular Cayley graph with valency $3$. Then $\Gamma$ is isomorphic to one of the following graphs.
  \begin{itemize}
    \item the complete graph $K_{4}$,
    \item the complete bipartite graph $K_{3,3}$,
    \item the cube $Q_3$,
    \item the Heawood graph $IG(7,3,1)$,
    \item the Pappus graph $IG(AG(2,3)\setminus pc)$.
  \end{itemize}
\end{thm}

\section{Distance-regular graphs with valency $4$}\label{sec:4}
The feasible intersection arrays for distance-regular graphs with valency four were determined by Brouwer and Koolen \cite{BK}. In Table \ref{tabledrgvalency4}, we give an overview of these intersection arrays and corresponding graphs, and indicate which of these is a Cayley graph, like in the previous section. Note that for each intersection array in the table there is a unique distance-regular graph, except possibly for the last array, which corresponds to the incidence graph of a generalized hexagon of order $(3,3)$.

\begin{table}[h]
\begin{center}
\begin{tabular}{ l r r c l c l}
\hline
   Intersection array & $n$ & $d$ & $g$& Name & Cayley & Reference\\
  \hline
  \{4;1\} & 5 & 1 &3& $K_{5}$ & Yes & Sec.~\ref{complete bipartite minus matching}\\
  \{4,1;1,4\} &  6 & 2 &3& $K_{2,2,2}$ & Yes & Sec.~\ref{complete bipartite minus matching}\\
  \{4,3;1,4\} &  8 & 2 &4& $K_{4,4}$ & Yes & Sec.~\ref{complete bipartite minus matching}\\
  \{4,2;1,2\} &  9 & 2 &3& $P(9) \sim H(2,3)$ & Yes & Sec.~\ref{Paley}\\
  \{4,3,1;1,3,4\} &  10 & 3 &4& $K^{*}_{5,5}$ & Yes & Sec.~\ref{complete bipartite minus matching}\\
  \{4,3,2;1,2,4\} &  14 & 3 &4& $IG(7,4,2)$ & Yes & Sec.~\ref{symmetricdesign}\\
  \{4,2,1;1,1,4\} &  15 & 3 &3& L(Petersen) & No & \cite[Prop.~5.1]{ADJ}\\
  \{4,3,2,1;1,2,3,4\} &  16 & 4 &4& $Q_{4}$ & Yes & Sec.~\ref{Cubes}\\
  \{4,2,2;1,1,2\} &  21 & 3 &3& L(Heawood) & Yes & \cite[Ex.~5.7]{ADJ}\\
   \{4,3,3;1,1,4\} &  26 & 3 &6& $IG(13,4,1)$ & Yes & Sec.~\ref{symmetricdesign}\\
   \{4,3,3,1;1,1,3,4\} &  32 & 4 &6& $IG(A(2,4)\setminus pc)$ & Yes & Prop.~\ref{prop:affineplane}\\
   \{4,3,3;1,1,2\} &  35 & 3 &6& $O_4$ & No & Sec.~\ref{Oddgraph}\\
   \{4,2,2,2;1,1,1,2\} &  45 & 4 &3& L(Tutte's 8-cage)
   & No & Prop.~\ref{tuttecoxeter}\\
   \{4,3,3,2,2,1,1; &  70 & 7 &6& $DO_4$ & No & Prop.~\ref{double Odd}\\
    \hspace{.5cm}1,1,2,2,3,3,4\} & & & && &\\
   \{4,3,3,3;1,1,1,4\} &  80 & 4 &8& $IG(GQ(3,3))$ & No & Prop.~\ref{IG(GQ)}\\
   \{4,2,2,2,2,2; &  189 & 6 &3& L(Tutte's 12-cage)
   & No & Prop.~\ref{CL(GD(2,1))}\\
     \hspace{.5cm}1,1,1,1,1,2\} & & & && &\\
   \{4,3,3,3,3,3; &  728 & 6 &12& $IG(GH(3,3))$ & No & Prop.~\ref{IG(GH)}\\
     \hspace{.5cm}1,1,1,1,1,4\}  && & && &\\
  \hline
\end{tabular}
 \caption{Distance-regular graphs with valency $4$}\label{tabledrgvalency4}
\end{center}
\end{table}

In \cite{ADJ}, distance-regular Cayley graphs with least eigenvalue $-2$ were studied. It was, among others, shown that the line graph of the Petersen graph is not a Cayley graph (see \cite[Prop.~5.1]{ADJ}), and that the line graph of Tutte's $8$-cage is not a Cayley graph (see Section \ref{sec:erratum}). On the other hand, it was shown that the line graph of the Heawood graph is a Cayley graph, over $\mathbb{Z}_7 \rtimes \mathbb{Z}_3$ (see \cite[Ex.~5.7]{ADJ}). In Proposition~\ref{CL(GD(2,1))}, we obtained that the line graph of Tutte's $12$-cage is not a Cayley graph. We can therefore conclude this section with the following result.

\begin{thm}
Let $\Gamma$ be a distance-regular Cayley graph with valency $4$. Then $\Gamma$ is isomorphic to one of the following graphs.
  \begin{itemize}
    \item the complete graph $K_{5}$,
    \item the octahedron graph $K_{2,2,2}$,
    \item the complete bipartite graph $K_{4,4}$,
    \item the Paley graph $P(9)$,
    \item the complete bipartite graph $K_{5,5}$ minus a complete matching,
    \item the incidence graph of the $2$-$(7,4,2)$ design,
    \item the cube graph $Q_4$,
    \item the line graph of the Heawood graph,
    \item the incidence graph of the projective plane over $GF(3)$,
    \item the incidence graph of the affine plane over $GF(4)$ minus a parallel class of lines.
  \end{itemize}
\end{thm}

\section{Distance-regular graphs with valency $5$}\label{sec:5}

In Table \ref{tabledrgvalency5}, we list all known putative intersection arrays for distance-regular graphs with valency $5$. We expect that this list is complete, but there is no proof for this. It contains all intersection arrays with diameter at most $7$. 
This can be derived from the tables in \cite{tables} and \cite{DKT}.
All of the graphs in the table are unique, given their intersection arrays, except possibly the incidence graph of a generalized hexagon of order $(4,4)$ (the last case).

\begin{table}[h]
\begin{center}
\begin{tabular}{ l r r c l c l}
  \hline
   Intersection array & $n$ & $d$ &$g$& Name & Cayley & Reference\\
  \hline
 \{5;1\} & 6 & 1 &3& $K_{6}$ & Yes & Sec.~\ref{complete bipartite minus matching}\\
 \{5,4;1,5\} &  10 & 2 &4& $K_{5,5}$ & Yes & Sec.~\ref{complete bipartite minus matching}\\
 \{5,2,1;1,2,5\} &  12 & 3 &3& Icosahedron & Yes & Folklore\\
 \{5,4,1;1,4,5\} &  12 & 3 &4 &$K^{*}_{6,6}$ & Yes & Sec.~\ref{complete bipartite minus matching}\\
 \{5,4;1,2\} &  16 & 2 &4& Folded $5$-cube & Yes & Sec.~\ref{Cubes}\\
 \{5,4,3;1,2,5\} &  22 & 3 &4& $IG(11,5,2)$ & Yes & Sec.~\ref{symmetricdesign}\\
 \{5,4,3,2,1;1,2,3,4,5\} &  32 & 5 &4& $Q_{5}$ & Yes & Sec.~\ref{Cubes}\\
 \{5,4,1,1;1,1,4,5\} &  32 & 4 &5& Armanios-Wells & Yes & Prop.~\ref{Armanios-Wells}\\
 \{5,4,2;1,1,4\} &  36 & 3 &5& Sylvester & No & Prop.~\ref{Sylvester}\\
  \{5,4,4;1,1,5\} &  42 & 3 &6& $IG(21,5,1)$ & Yes & Sec.~\ref{symmetricdesign}\\
  \{5,4,4,1;1,1,4,5\} &  50 & 4 &6& $IG(A(2,5)\setminus pc)$ & Yes & Prop.~\ref{prop:affineplane}\\
  \{5,4,4,3;1,1,2,2\} &  126 & 4 &6& $O_{5}$ & No & Sec.~\ref{Oddgraph}\\
  \{5,4,4,4;1,1,1,5\} &  170 & 4 &8& $IG(GQ(4,4))$ & No & Prop.~\ref{IG(GQ(4,4))}\\
  \{5,4,4,3,3,2,2,1,1; &  252 & 9 &6& $DO_{5}$ & No & Prop.~\ref{double Odd}\\
   \hspace{.5cm}1,1,2,2,3,3,4,4,5\}  && & && &\\
  \{5,4,4,4,4,4; &  2730 & 6 &12& $IG(GH(4,4))$ & No & Prop.~\ref{IG(GH)}\\
  \hspace{.5cm}1,1,1,1,1,5\}  && & && &\\
  \hline
\end{tabular}
 \caption{Distance-regular graphs with valency $5$}\label{tabledrgvalency5}
\end{center}
\end{table}

It is well-known that the icosahedron is a Cayley graph. By using {\sf GAP} \cite{GAP} and similar codes as in \cite[p.3]{AJ}, we checked that we can indeed describe the icosahedron as a Cayley graph over the alternating group $Alt(4)$, with connection set $S=\{(123),(132),(12)(34),(134),(143)\}$.
According to Miklavi\v{c} and Poto\v{c}nik \cite{MP2}, the icosahedron is the smallest distance-regular Cayley graph over a non-abelian group, if we exclude cycles and the graphs from Section \ref{complete bipartite minus matching}.

Also the Armanios-Wells graph is a Cayley graph. As far as we know, this was not known before.

Indeed, let $G$ be the group generated by elements $g_i$, with $i=1,2,3,4,$ each of order $2$, such that $[g_i,g_j]$ is the same element, $a$ say, for all $i \neq j$. This group is isomorphic to $(\mathbb{Z}_{2} \times Q_{8}) \rtimes \mathbb{Z}_{2}$, where $Q_8$ is the group of quaternions. Now let $S=\{g_1,g_2,g_3,g_4,g_1g_2g_3g_4\}$. Then it is not hard to check that the Cayley graph $Cay(G,S)$ is distance-regular with the same intersection array as the Armanios-Wells graph $\G$, and hence that it must be the latter. In order to indeed check this, it is useful to know that $\G$ is an antipodal double cover with diameter $4$, and that in this case $S_4=\{a\}$, and consequently $S_3=Sa$ (see Section \ref{sec:normal}). We double-checked this with {\sf GAP} \cite{GAP}, and thus we have the following.

\begin{prop} \label{Armanios-Wells}
The Armanios-Wells graph is a Cayley graph over $(\mathbb{Z}_{2} \times Q_{8}) \rtimes \mathbb{Z}_{2}$.
\end{prop}

A few more observations that we should make are the following. The center of $G$ equals $\seq{a}$, which is of order $2$. The quotient $G/\seq{a}$ is isomorphic to the elementary abelian $2$-group $\mathbb{Z}_{2}^4$, which leads to the well-known description of the quotient graph --- the folded $5$-cube --- as a Cayley graph (see Section \ref{Cubes}).

The group $G$ has a normal subgroup $\seq{g_1g_2,g_2g_3,g_3g_1}$, which is isomorphic to $Q_{8}$. This gives rise to an equitable partition of $\G$ into $4$ cocliques of size $8$.

In addition, the normal subgroup $\seq{g_1g_2,g_2g_3,g_3g_1,g_4}$ is isomorphic to $\mathbb{Z}_{2} \times Q_{8}$, which gives an equitable partition of $\G$ into two $1$-regular induced subgraphs. Together these form a matching, and removing the edges of this matching results in a bipartite $4$-regular graph. This turns out to be the incidence graph of the affine plane of order $4$ minus a parallel class (see Section \ref{affineplane} and Table \ref{tabledrgvalency4}). Alternatively, we obtain that the latter is isomorphic to the Cayley graph $Cay(G,\{g_1,g_2,g_3,g_4\})$.

The remaining intersection array in Table \ref{tabledrgvalency5} is that of the Sylvester graph. This graph has distinct eigenvalues $5,2,-1$, and $-3$ and full automorphism group $Sym(6) \rtimes \mathbb{Z}_{2}$ \cite[p.~394]{BCN}.

\begin{prop} \label{Sylvester}
The Sylvester graph is not a Cayley graph.
\end{prop}
\begin{proof}
Suppose that the Sylvester graph $\G$ is a Cayley graph $Cay(G,S)$, then $|G|=36$ and $|S|=5$. Because $\G$ has girth $5$, the group $G$ is non-abelian by Lemma \ref{abelian}. It is known that there are $10$ non-abelian groups of order $36$, of which two do not have a normal subgroup of order $9$; these are $\mathbb{Z}_{3} \times Alt(4)$ and $(\mathbb{Z}_{2} \times \mathbb{Z}_{2}) \rtimes \mathbb{Z}_{9}$.

If $G$ is the latter group (and contains elements of order $9$), then it has automorphisms of order $9$. This contradicts the fact that the full automorphism group of $\G$ equals $Sym(6) \rtimes \mathbb{Z}_{2}$.

Next, we will also show that $G$ cannot be $\mathbb{Z}_{3} \times Alt(4)$, and hence that $G$ must have a normal subgroup of order $9$. Indeed, suppose that $G$ equals $\mathbb{Z}_{3} \times Alt(4)$. The center of this group is isomorphic to $\mathbb{Z}_{3}$, say $Z(G)=\seq{c}$, with $c$ of order $3$. Moreover, $G$ has a normal subgroup $H$ isomorphic to $Alt(4)$ (with cosets $H,Hc,Hc^2$ that form an equitable partition of $\G$).

Now suppose that $hc^i\in S$ for some $h \in H$ and $i=0,1,2$. Then the order of $h$ must be $2$, for if it were $3$ (or $1$, the only other options), then $e \sim hc^i \sim (hc^i)^{2} \sim (hc^i)^{3}=e$, which contradicts the fact that $\G$ has girth $5$. Moreover, if $h \in S$, then $hc$ and $hc^{2}=(hc)^{-1}$ are not in $S$ because that would imply that $e \sim hc \sim c \sim hc^{2} \sim e$, which again gives a contradiction.

Because $Alt(4)$ has only three involutions, there are also only three involutions $h_1,h_2,$ and $h_3$, say, in $H$. Thus, it follows without loss of generality that $S=\{h_1,h_2c,h_2c^2,h_3c,h_3c^2\}$. However, now $e \sim h_2c \sim h_3h_2 \sim h_2c^2 \sim e$, which gives the final contradiction, and hence $G$ cannot be $\mathbb{Z}_{3} \times Alt(4)$.

Thus, the group $G$ has a normal subgroup $N$ of order $9$. The four cosets of $N$ form an equitable partition of $\G$ with quotient matrix
$$\begin{bmatrix}n_1&n_2&n_3&n_4\\n_2&n_1&n_4&n_3\\n_3&n_4&n_1&n_2\\n_4&n_3&n_2&n_1\end{bmatrix},$$
for certain $n_1,n_2,n_3,n_4$ summing to $5$, and because $\G$ is connected, at most one of $n_2,n_3,n_4$ can be 0. Now the quotient matrix has eigenvalues $n_1+n_2+n_3+n_4$, $n_1+n_2-n_3-n_4$, $n_1-n_2+n_3-n_4$, and $n_1-n_2-n_3+n_4$. Because $\G$ has no eigenvalues $3$ and $1$, it follows that $n_1=0,n_2=1,n_3=2,$ and $n_4=2$, up to reordering of the latter three (we omit the easy but technical details).

So there is one coset that intersects $S$ in $n_2=1$ element. Let us call this element $a$, then clearly $O(a)=2$, and the subgroup $N \seq{a}$ is a normal subgroup (of index $2$). Given the quotient matrix, it follows easily that every vertex in the coset $Na$ except $a$ itself is at distance $2$ from $e$.

Now we claim that $a$ is the only involution in $N \seq{a}$. Clearly there are no involutions in $N$ because it has order $9$. Every other element in $Na$ is at distance $2$ from $e$, and hence can be written as $s_1s_2$ for some $s_{1},s_{2} \in S$. Suppose now that $O(s_{1}s_{2})=2$. Then $e \sim s_{2} \sim s_{1}s_{2} \sim s_{2}s_{1}s_{2} \sim e$, a contradiction since the girth of $\G$ is $5$, and we proved our claim.

Now suppose that $s \in S$, with $s \neq a$. Then $s^{-1}as \in N \seq{a}$ since $N \seq{a}$ is a normal subgroup. Because $O(s^{-1}as)=2$, it follows from our above claim that $s^{-1}as=a$. Thus, $sa=as$ and $e \sim a \sim sa=as \sim s \sim e$, which is again a contradiction to the girth of $\G$, and which completes the proof.
\end{proof}

Now we can conclude this section with the following proposition.
\begin{prop} Let $\Gamma$ be a distance-regular Cayley graph with valency $5$, with one of the intersection arrays in Table \ref{tabledrgvalency5}\footnote{Currently, these are the only known putative intersection arrays for distance-regular graphs with valency 5}. Then $\Gamma$ is isomorphic to one of the following graphs.
  \begin{itemize}
    \item the complete graph $K_{6}$,
    \item the complete bipartite graph $K_{5,5}$,
    \item the icosahedron,
    \item the complete bipartite graph $K_{6,6}$ minus a complete matching,
    \item the folded $5$-cube,
    \item the incidence graph of the $2$-$(11,5,2)$ design,
    \item the cube graph $Q_5$,
    \item the Armanios-Wells graph,
    \item the incidence graph of the projective plane over $GF(4)$,
    \item the incidence graph of the affine plane over $GF(5)$ minus a parallel class of lines.
  \end{itemize}
\end{prop}

\section{Distance-regular graphs with girth $3$ and valency $6$ or $7$}\label{sec:67}

Hiraki, Nomura, and Suzuki \cite{HNS} determined the feasible intersection arrays of all distance-regular graphs with valency at most $7$ and girth $3$ (i.e., with triangles). Besides the ones with valency at most $5$ that we have encountered in the previous sections, these are listed in Table \ref{tabledrgvalency6}. For each of the intersection arrays $\{6,3;1,2\}$ and $\{6,4,4;1,1,3\}$, there are exactly two distance-regular graphs (as mentioned in the table). For all others, except possibly the last one with valency $6$, the graphs in the table are unique, given their intersection arrays. For this last case, it is unknown whether the generalized hexagon of order $(3,3)$ is unique.

\begin{table}[h]
\begin{center}
\begin{tabular}{ l r r c l c l}
  \hline
   Intersection array & $n$ & $d$ &$g$& Name & Cayley & Reference\\
  \hline
 \{6;1\} & 7 & 1 &3& $K_{7}$ & Yes & Sec.~\ref{complete bipartite minus matching}\\
 \{6,1;1,6\} &  8 & 2 &3& $K_{2,2,2,2}$ & Yes & Sec.~\ref{complete bipartite minus matching}\\
 \{6,2;1,6\} &  9 & 2 &3& $K_{3,3,3}$ & Yes & Sec.~\ref{complete bipartite minus matching}\\
 \{6,2;1,4\} &  10 & 2 &3& $T(5)$ & No & Sec.~\ref{Oddgraph}\\
 \{6,3;1,3\} &  13 & 2 &3& $P(13)$ & Yes & Sec.~\ref{Paley}\\
 \{6,4;1,3\} &  15 & 2 &3& $\overline{T(6)}\sim GQ(2,2)$ & No & Sec.~\ref{Oddgraph}\\
 \{6,3;1,2\} &  16 & 2 &3& $L_{2}(4)$, Shrikhande & Yes & Sec.~\ref{Cubes}\\
 \{6,4,2;1,2,3\} &  27 & 3 &3& $H(3,3)$ & Yes & Sec.~\ref{Cubes}\\
 \{6,4,2,1;1,1,4,6\} &  45 & 4 &3& halved Foster & No & Cor.~\ref{halffoster}\\
 \{6,3,3;1,1,2\} &  52 & 3 &3& L($IG(13,4,1)$) & No & Prop.~\ref{LinePP3}\\
 \{6,4,4;1,1,3\} &  63 & 4 &3& $GH(2,2)$ ($2\times$)& No & Prop.~\ref{CL(GD(2,1))}\\
 \{6,3,3,3;1,1,1,2\} &  160 & 4 &3& L($IG(GQ(3,3))$) & No & Prop.~\ref{LineGQ3}\\
 \{6,3,3,3,3,3;1,1,1,1,1,2\} &  1456 & 6 &3& L($IG(GH(3,3))$) & No & Prop.~\ref{LineGH3}\\
 \hline
  \{7;1\} & 8 & 1 &3& $K_{8}$ & Yes & Sec.~\ref{complete bipartite minus matching}\\
 \{7,4,1;1,2,7\} &  24 & 3 &3& Klein & Yes & Prop.~\ref{Klein}\\
 \hline
\end{tabular}
 \caption{Distance-regular graphs with girth $3$ and valency $6$ or $7$}\label{tabledrgvalency6}
\end{center}
\end{table}

What remains is to consider the Klein graph. We observe that this is a Cayley graph on the symmetric group $Sym(4)$. Indeed, one can check\footnote{We double-checked this with {\sf GAP} \cite{GAP}} that with $$S=\{(123),(132),(12)(34),(13),(14),(1234),(1432)\},$$ the Cayley graph $Cay(Sym(4),S)$ is a distance-regular antipodal $3$-cover of $K_8$, and hence it must be the Klein graph. We note that in this case the set $S_3=\{(124),(142)\}$, and despite the fact that $N_3=S_3 \cup \{e\}$ is not a normal subgroup, its right cosets form an equitable partition (with quotient $K_8$, of course); cf.~Section \ref{sec:normal}. We thus have the following.

\begin{prop} \label{Klein}
The Klein graph is a Cayley graph over $Sym(4)$.
\end{prop}

We also note that the normal subgroup $\{e,(12)(34),(13)(24),(14)(23)\}$ gives an equitable partition into $6$ parts, with each coset inducing a matching (which together gives a perfect matching). More interesting is the (normal) alternating subgroup $Alt(4)$, which gives an equitable partition into two parts. On each part, the induced subgraph is the truncated tetrahedron, which is thus a Cayley graph $Cay(Alt(4),\{(123),(132),(12)(34)\})$. This is also the line graph of a bipartite biregular graph on $4 +{{4}\choose{2}}$ vertices with valencies $3$ and $2$, respectively (the Pasch configuration), and a subgraph of the icosahedron; cf.~Section \ref{sec:5}.

We conclude with the following proposition.

\begin{prop} Let $\Gamma$ be a distance-regular Cayley graph with girth $3$ and valency $6$ or $7$. Then $\Gamma$ is isomorphic to one of the following graphs.
  \begin{itemize}
    \item the complete graph $K_{7}$,
    \item the complete graph $K_{8}$,
    \item the complete multipartite graph $K_{2,2,2,2}$,
    \item the complete multipartite graph $K_{3,3,3}$,
    \item the Paley graph $P(13)$,
    \item the lattice graph $L_2(4)$,
    \item the Shrikhande graph,
    \item the Hamming graphs $H(3,3)$,
    \item the Klein graph.
  \end{itemize}
\end{prop}


\section*{Acknowledgements}
\noindent The research of Mojtaba Jazaeri was in part supported by a grant from School of Mathematics, Institute for Research in Fundamental Sciences (IPM) (No. 95050039). We thank Sasha Gavrilyuk, who, in the final stages of writing this paper, at a conference in Plze\v{n}, pointed us to \cite{TTvM} for what he called Higman's method applied to generalized polygons.


\end{document}